\documentclass[12pt]{amsart}

\usepackage{amsthm, amssymb, amstext, amscd, amsfonts, amsxtra, latexsym, amsmath, comment, 
 mathrsfs, esint, stmaryrd,cancel}
\usepackage{color}
\usepackage[svgnames]{xcolor}
\usepackage{fancybox}
\usepackage{fullpage}
\usepackage[english]{babel}
\usepackage[latin1]{inputenc}
\usepackage{verbatim}
\usepackage{color}
\usepackage[all]{xy}
\usepackage[mathscr]{eucal}
\usepackage{enumitem}

\numberwithin{equation}{section}
\newtheorem{theorem}{Theorem}
\numberwithin{theorem}{section}
\newtheorem{lemma}[theorem]{Lemma}
\newtheorem{proposition}[theorem]{Proposition}

\theoremstyle{remark}
\newtheorem*{remark}{Remark}

\newcommand{\Z}{\mathbb{Z}}
\newcommand{\N}{\mathbb{N}}

\renewenvironment{proof}[1][Proof]{\begin{trivlist} \item[\hskip \labelsep {\bfseries #1:}]}{\qed\end{trivlist}}

\title{Improved bounds for Fourier coefficients of Siegel modular forms}
\author{Kathrin Bringmann} 
\address{Mathematical Institute\\University of
Cologne\\ Weyertal 86-90 \\ 50931 Cologne \\Germany}
\email{kbringma@math.uni-koeln.de}
\thanks{ The research of the author was supported by the Alfried Krupp Prize for Young University Teachers of the Krupp foundation and the research leading to these results has received funding from the European Research Council under the European Union's Seventh Framework Programme (FP/2007-2013) / ERC Grant agreement n. 335220 - AQSER}

\begin{document}
\maketitle

\section{Introduction and statement of results}
The goal of this paper is to improve existing bounds for Fourier coefficients of higher genus Siegel modular forms of small weight.

To recall what is known for genus 1, let $\Delta(\tau):=q\prod_{n=1}^{\infty}(1-q^n)^{24}$ $(q:=e^{2\pi i\tau})$ be the classical $\Delta$-function and denote by $\tau(n)$ its Fourier coefficients. The {\it Ramanujan conjecture} states that, for $p$ prime,
$$
|\tau(p)|\leq2p^{\frac{11}{2}}.
$$
This conjecture has been generalized for general, positive integral weight modular forms. The so-called {\it Ramanujan-Petersson conjecture} states that if $f(\tau)=\sum_{n=1}^{\infty}a(n)q^n$ is a weight $k$ cusp form on a congruence subgroup, then, as $n\rightarrow \infty$,
\begin{equation}\label{RP}
	a(n)\ll_{\varepsilon,f} n^{\frac{k-1}{2}+\varepsilon}\qquad\qquad (\varepsilon>0).
\end{equation}
The estimate \eqref{RP} follows from Deligne's proof of the Weil conjectures \cite{De,DS}, using highly complicated methods from algebraic geometry.

\indent
There are many related conjectures for more complicated types of automorphic forms. In this paper, we consider the case of Siegel modular form of genus $g>1$. For this, let $F$ be a cusp form of weight $k\in\mathbb{N}$ with respect to the Siegel modular group
         $\Gamma_g:= \text{Sp}_g(\mathbb{Z})\subset \text{GL}_{2g}(\mathbb{Z})$ with Fourier coefficients
         $a(T)$, where $T$ is a positive definite symmetric half-integral $g \times g$ matrix.
                  Then a conjecture of Resnikoff and Salda\~na \cite{RS} says that
\begin{displaymath}
a(T) \ll_{\varepsilon,F} \det (T)^{\frac{k}{2}- \frac{(g+1)}{4}+  \varepsilon}   \qquad  (\varepsilon >0).
\end{displaymath}
For $g=1$ this is exactly the Ramanujan-Petersson conjecture. For higher genus $g$, however, there are counterexamples coming from lifts (cf. \cite{K2}).

 For $k>g+1$, the best known estimate is
 \begin{eqnarray} \label{best}
a(T) \ll_{\varepsilon, F} \det (T)^{\frac{k}{2} - c_g + \varepsilon }     \qquad (\varepsilon >0),
\end{eqnarray}
where 
\begin{displaymath}
c_g:= \left\{
\begin{array}{lll}
\frac{13}{36} & \text{ if } g=2 & \text{\cite{K2}},\\[1ex]
\frac{1}{4} & \text{ if } g=3 & \text{\cite{Bre}},\\[1ex]
\frac{1}{2g} + \left(1- \frac{1}{g} \right)\alpha_g & \text{ if } g>3& \text{\cite{BK}.}
\end{array}   \right.
\end{displaymath}     Here
\begin{displaymath}
\alpha_g^{-1}:= 4(g-1)+ 4 \left[ \frac{g-1}{2} \right]     + \frac{2}{g+2}.
\end{displaymath}
In \cite{Bri} and \cite{BYang} it was shown that (\ref{best}) still holds for $k=g+1 \text{ and }k=g$, respectively. Moreover, for $ (g+3)/2 < k < g$, we have \cite{BYang} 
\begin{equation}\label{TKbound}
a(T) \ll_{\varepsilon ,F} \det (T)^{\frac{k}{2} - \left( 1-\frac{1}{g}\right) \alpha_g +\varepsilon} .
\end{equation}
\indent
In this paper we improve \eqref{TKbound} and obtain

\begin{theorem}\label{maintheorem}
We have for $g/2+1<k<g$
$$
a(T) \ll_{\varepsilon ,F} \operatorname{det}(T)^{\frac{k}{2}+ \frac{g-k}{2g(g-2)}-\frac{1}{2g} - \left(1-\frac{1}{g}\right)\alpha_g+\varepsilon}.
$$
\end{theorem}
\begin{remark}
Theorem \ref{maintheorem} is indeed an improvement since $$\frac{g-k}{2g(g-2)}-\frac{1}{2g}<0.$$
\end{remark}

Our proof follows the idea of \cite{BK} using a Jacobi decomposition of Siegel modular forms. Our main achievement is an improved bounds for Kloosterman sums.

The paper is organized as follows. In Section 2 we recall basic facts about Jacobi forms and their relation to Siegel modular forms. In Section 3 we bound higher dimensional Kloosterman sums. Section 4 is devoted estimating coefficients of Poincar\'e series, in Section 5 we then conclude our main theorem.

\section*{Acknowledgments}
The author thanks Winfried Kohnen and Mike Woodbury for comments on an earlier version of this paper.

\section{Preliminaries}
\subsection{Basic facts on Jacobi forms}

 Here we recall some basic facts about Jacobi cusp forms; for details we refer the reader to \cite{EZ} and \cite{Zi}.
  The Jacobi group
       \begin{math}
       \Gamma_{g}^J   := \text{SL}_2 (\mathbb{Z})\ltimes (\mathbb{Z}^g \times \mathbb{Z}^g)     \end{math}
         acts on $\mathbb{H}  \times \mathbb{C}^{g}$ in the usual way by ($\left(\begin{smallmatrix} a&b \\c&d \end{smallmatrix} \right) \in \text{SL}_2 (\mathbb{Z}), (\lambda ,\mu )\in \mathbb{Z}^g$)
\begin{displaymath}
\left(
\left(\begin{matrix} a &b \\ c &d \end{matrix} \right)
 ,(\lambda, \mu)\right)  \circ (\tau,z):=
\left( \frac{a \tau+b}{c \tau+d},
\frac{z+ \lambda \tau+ \mu}{c \tau +d}\right).
\end{displaymath}
Note that throughout vectors are viewed as columns unless noted otherwise.
        Let $k \in \mathbb{N}$, $m$ be a positive definite symmetric half-integral $g \times g$ matrix,
          $\gamma= \left(  \left(
         \begin{smallmatrix} a & b \\ c & d \end{smallmatrix}
          \right) ,
         (\lambda, \mu)\right) \in \Gamma_{g}^J
       $, and
$\phi : \mathbb{H} \times \mathbb{C}^{g} \to \mathbb{C}$.
Then we define  the following {\it Jacobi slash action}
 \begin{multline*}
\phi|_{k,m} \gamma (\tau,z):= (c \tau+d)^{-k} 
e\left(-c \left(c \tau+d)^{-1}  m[z+ \lambda \tau+ \mu]
+ m[\lambda]\tau + 2  \lambda^T m z \right)\right)
 \phi(\gamma \circ (\tau,z)),
\end{multline*}
where
\begin{math}
e(w):=e^{2 \pi i w} \quad \left(\,\forall \, w \in \mathbb{C}\right),
\end{math}
and where $A[B]:= B^T AB$ for matrices $A$ and $B$ of compatible sizes.

 A holomorphic function  $\phi : \mathbb{H}   \times \mathbb{C}^{g} \to \mathbb{C}$ is called a {\it Jacobi cusp form} of weight $k$
  and index $m$ with respect to $\Gamma_g^J$, if, for all  $\gamma \in \Gamma_g^J$,   we have
  \begin{math} \phi|_{k,m} \gamma =\phi,\end{math} and $\phi$
   has a Fourier expansion of the form
\begin{displaymath}
        \phi(\tau,z)=
 \sum_{D>0}
c(n,r)e\left(n \tau +r^T z \right),
        \end{displaymath}
        where   $D:= \det   \left(\begin{smallmatrix}2n&r^T\\  r &2m\end{smallmatrix}\right) $
        with $n \in \N$ and $r \in \Z^g$.
We denote by $J_{k,m}^{\text{cusp}}$  the vector space of Jacobi cusp forms. 

The space $J_{k,m}^{\text{cusp}}$ is   a finite dimensional Hilbert  space with  the Petersson scalar product
\begin{eqnarray*}
\left<\phi, \psi \right>  :=
\int_{\Gamma_g^J \backslash \mathbb{H}\times \mathbb{C}^{g}}
\phi (\tau,z)\overline{\psi (\tau,z)} 
 \exp\left(
 - 4 \pi m[y]\cdot v^{-1}\right) v^k dV_g^J,
\end{eqnarray*}
where $dV_g^J := v^{-g-2} du dv dx dy$, $\tau= u + iv$,
and $ z=x+iy$.
\subsection{Jacobi Poincar\'e series}
We next recall certain Jacobi Poincar\'e series, as considered in \cite{BK}. For $ n \in \Z$, $r \in \Z^g$, and $m$ a positive definite symmetric half-integral $g \times g$ matrix such that $ 4 n > m^{-1} [r]$, define a  {\it Poincar\'e series of exponential type} by
\begin{equation}\label{JacPo}
P_{k, m; (n,r)}(\tau,z) := \sum_{\gamma \in \Gamma_{g,\infty}^{J}\big\backslash \Gamma_{g}^{J}} e^{n,r} \bigg|_{k,m}\gamma(\tau,z),
\end{equation}
where $e^{n,r}(\tau,z):= e^{2\pi i(n\tau + r^T z)}$ and 
$
\Gamma_{g,\infty}^{J} := \{ ((\begin{smallmatrix} 1 & n \\ 0& 1\end{smallmatrix}), (0, \mu)) | n \in \Z, \mu \in \Z^{g}\}
$
is the stabilizer group of $e^{n,r}$. For $ k > g+2$, $P_{k,m; (n,r)} \in  J_{k,m}^{\text{cusp}}$ 
and the Petersson coefficient formula holds $( \phi \in J_{k,m}^{\text{cusp}}$ with Fourier coefficients $c_{\phi})$,
\begin{equation}\label{Pet}
\left< \phi, P_{k,m;(n,r)}\right> = \lambda_{k,m,D} c_{\phi}(n,r),
\end{equation}
where 
$$
\lambda_{k,m,D} := 2^{-\frac{g}{2}}\Gamma\left(k-\frac{g}{2}-1\right) (2\pi)^{-k + \frac{g}{2} +1} \det(2m)^{k - \frac{g+3}{2}} D^{-k+ \frac{g}{2}+1}.
$$

For $k \leq g+2$ the Poincar\'e series \eqref{JacPo} diverge. However there is a way to analytically continue them, using the  so-called Hecke trick. We denote the corresponding functions again by $P_{k,m;(n,r)}$. We have \cite{BK, Bri, BYang}:
\begin{proposition}\label{Petprop}
For $k>g/2+2$, the functions $P_{k,m;(n,r)}$ are elements of $J_{k,m}^{\operatorname{cusp}}$. We have the Fourier expansions
$$
P_{k,m;(n,r)}(\tau,z)=\sum_{\substack{ n' \in \mathbb{Z}, r' \in \mathbb{Z}^g \\ D'>0}}g_{k,m;(n,r)}^\pm (n',r')e\left(n'\tau+r'^Tz\right),
$$
where $D':=\operatorname{det}\left(\begin{smallmatrix}2n'&r'^T\\  r' &2m\end{smallmatrix}\right) $ and
\begin{equation*}
g_{k,m;(n,r)}^\pm(n',r'):=g_{k,m;(n,r)}(n',r')+(-1)^k g_{k,m;(n,r)}(n', -r')
\end{equation*}
with
\begin{multline}\label{defineg}
g_{k,m;(n,r)}(n',r'):=\delta_m(n,r,n',r')+2\pi i^k\det(2m)^{-\frac{1}{2}}\left(\frac{D'}{D}\right)^{\frac{k}{2}-\frac{g}{4}-\frac{1}{2}}\\\times \sum_{c\geq 1} e_{2c}\left(r^T m^{-1}r'\right) H_{m,c}(n,r,n',r')J_{k-\frac{g}{2}-1}\left(\frac{2\pi\sqrt{DD'}}{\det (2m)c}\right)c^{-\frac{g}{2}-1}.
\end{multline}
\noindent
Here $e_c (x):= e^{\frac{2\pi i x}{c}}$,
$$
\delta_m(n,r,n',r'):=
\begin{cases}
1&\text{if }D'=D \text{ and } r'- r \in 2m\Z^g,\\
0&\text{otherwise}
\end{cases}
$$ 
and the Kloosterman sums
$$
H_{m,c}(n,r,n',r'):=\sum_{\substack{\lambda \pmod{c}\\d\pmod{c}^*}}e_c\left((m[\lambda]+r^T\lambda+n)\overline{d}+n'd+r'^T\lambda\right),
$$
where by $\lambda \pmod{c}$, we mean that all components run $\pmod{c}$ and $d\pmod{c}^*$ sums only over $d\pmod{c}$ which are coprime to $c$.
Moreover formula \eqref{Pet} holds.
\end{proposition}

\noindent {\it Remark.} Note that in \cite{BK} the Kloosterman sums have a slightly different normalization.

Proposition \ref{Petprop} gives that for $k>g/2+2$ the $P_{k,m;(n,r)}$ are a generating system of $J_{k,m}^{\operatorname{cusp}}$. We easily obtain, just using the Cauchy-Schwarz inequality
\begin{lemma}
For $k>g/2+2$ and $\phi\in J_{k,m}^{\operatorname{cusp}}$ with Fourier coefficients $c_\phi (n,r)$, we have
$$
|c_\phi (n,r)|\ll_k\left|b_{n,r}\left(P_{k,m;(n,r)}\right)\right|^{\frac12}\frac{D^{\frac{k}{2}-\frac{g}{4}-\frac12}}{\operatorname{det}(2m)^{\frac{k}{2}-\frac14(g+3)}}\Vert\phi\Vert.
$$
\end{lemma}
Thus, to get bounds for the Fourier coefficients of Jacobi forms, one only has to bound the Fourier coefficients of the Poincar\'e series which are explicitly given in Proposition \ref{Petprop}. However, we also bound coefficients of Siegel modular forms, which requires estimating $\lVert \phi \rVert$. The connection between Siegel modular forms and Jacobi forms is described in the next subsection.

\subsection{Relation to Siegel modular forms}
Let $\mathbb{H}_g$ be the usual Siegel upper half space and write $Z \in \mathbb{H}_g$ as $Z= \left(\begin{smallmatrix} \tau & z^T \\ z & \tau' \end{smallmatrix}\right)$ with $\tau \in \mathbb{H}$, $ z \in \mathbb{C}^{g-1}$, and $\tau' \in \mathbb{H}_{g-1}$. Then $F \in S_k(\Gamma_g)$, the space of Siegel cusp forms of weight $k$ for $\Gamma_g$, has a so-called {\it Fourier Jacobi expansion} of the form
$$
F(Z) = \sum_{m > 0}\phi_m(\tau, z) e^{2\pi i \operatorname{tr}(m \tau')},
$$
where $\text{tr}$ denotes the trace of a matrix and 
where $m$ runs through all positive definite symmetric half-integral $(g-1) \times (g-1)$ matrices. It is well-known, that the coefficients of $\phi_m$ are Jacobi cusp forms. So bounds for the Fourier coefficients of Siegel modular forms follow from the understanding of the coefficients of Jacobi forms.
\section{Bounding Kloostermann sums}

A first step in bounding Fourier expansions of Poincar\'e series is to estimate certain higher-dimensional Kloosterman sums which occur when restricting 
the Fourier coefficients of Jacobi Poincar\'e series to the diagonal $(n',r')=(n,r)$. To be more precise, we set
\begin{equation*}
H_{m,c}^\pm(n,r):=H_{m, c}(n, r, n, \pm r).
\end{equation*}

To bound these, we require well-known evaluations of (generalized) {\it Gauss sums} 
\[
G(a,b;c):= \sum_{n \pmod{c}} e_c \left(an^2 +bn\right).
\]
\begin{lemma}\label{Gausslemma}
 Let $p$ be prime, $a,b\in \mathbb{Z}$, $\nu \in \mathbb{N}$, and $\alpha := \text{\emph{ord}}_p (a)$.
 \begin{enumerate}[leftmargin=*]
 \item[(1)] For $\alpha \geq \nu$, we have
 \[
 G\left(a,b; p^\nu \right) = \begin{cases} p^\nu & \text{if } b\equiv 0 \pmod{p^\nu}, \\ 0&\text{otherwise.} \end{cases}
 \]
 \item[(2)] For $0\leq \alpha <\nu$, $G(a,b;p^\nu) =0$ unless $b\equiv 0 \pmod{p^\alpha}$ in which case we have the following evaluations:
 \begin{enumerate}[leftmargin=5mm]
 \item[(i)] If $p\not =2$ and $b\equiv 0 \pmod{p^\alpha}$, then
 \[
 G\left( a,b;p^\nu \right) = p^{\frac{\alpha +\nu}{2}} \varepsilon_{p^{\nu -\alpha}} \left( \frac{a/p^\alpha}{p^{\nu -\alpha}} \right) e_{p^{\nu +\alpha}} \left( -b^2 \frac{\overline{4a}}{p^\alpha} \right) ,
 \]
 where $\overline{\ell}$ denotes the inverse of $\ell \pmod{p^{\nu +a}}$ and $\varepsilon_j = 1$ or $i$ depending on whether $j\equiv 1\pmod{4}$ or $j\equiv 3\pmod{4}$, respectively.
 \item[(ii)] If $p=2$ and $b\equiv 0 \pmod{p^\alpha}$, then $G( a,b;p^\nu)$ equals
 $$\begin{cases}
 2^\nu &\text{if }\alpha =\nu -1\text{ and }b\not \equiv 0\pmod{2^\nu}, \\
 {\tiny 2^{\frac{\nu +\alpha}{2}} \left( \tfrac{-2^{\nu -\alpha}}{a/2^\alpha} \right) \varepsilon_{\frac{a}{2^\alpha}} (1+i) e_{2^{\nu +\alpha +2}} \left( -b^2 \overline{\tfrac{a}{2^\alpha}} \right)}
  &\text{if }b\equiv 0\pmod{2^{\alpha +1}} \text{ and }\nu \equiv \alpha \pmod{2}, \\
 0 &\text{otherwise},
 \end{cases}$$
where $\overline{\ell}$ denotes the inverse of $\ell \pmod{2^{\nu +\alpha +2}}$.
 \end{enumerate}
 \end{enumerate}
\end{lemma}
We are now ready to bound the higher-dimensional Klosterman sums.
\begin{lemma}
\label{KloosLemma}
We have
\begin{equation*}
H^\pm_{m,c}(n,r)\ll \left(D,c\right)c^{\frac{g+1}{2}} \det(2m)^{\frac{1}{2}}.
\end{equation*}
\end{lemma}
\begin{proof}
Our proof closely follows the one in \cite{BK}. There it was shown on page 507 that, for $c=c_1 c_2$ with $(c_1 ,c_2)=1$, $$H^\pm_{m,c} (n,r) = H^\pm_{c_1 m,c_2} (n \overline{c_1},r)H^\pm_{c_2 m,c_1} (n \overline{c_2},r),$$ where $\overline{c_1}$ and $\overline{c_2}$ are inverses of $c_1$ and $c_2$ modulo $c_2$ and $c_1$, respectively. Thus we may assume that $c=p^\nu$ with $p$ prime and $\nu \in\mathbb{N}$ and for simplicity we for now restrict to $p\neq2$. The modifications required for $p=2$ follow along the same lines as in \cite{BK}.

Since a non-degenerate binary quadratic form over $\mathbb{Z}_p$ ($p\not =2$) is diagonalizable, we may assume that 
$m=\text{diag}(m_1 ,\dots ,m_g)$
is a diagonal matrix.
Set $\mu_j:=\operatorname{ord}_p(m_j)$ $(1\leq j\leq g)$. We assume without loss of generality that $\nu\leq \mu_j$ for $1\leq j \leq \ell$ and $\nu>\mu_j$ for $\ell+1\leq j \leq g$. Write $r=(r_1, \ldots, r_g)$. From (18) of \cite{BK}, we conclude that
\begin{equation*}
H_{m,p^{\nu}}^\pm(n,r)=\sum_{d\pmod{p^{\nu}}^*}e_{p^\nu}\left(n \left(d+\overline{d} \right) \right) \prod_{j=1}^{g}\sum_{\lambda_j \pmod{p^\nu}}e_{p^\nu}\left(\left(m_j\lambda_j ^2+r_j \lambda_j \right)\overline{d}\pm r_j \lambda_j \right).
\end{equation*}
The sum on $\lambda_j$ equals
$
G(m_j, r_j (\overline{d}\pm 1 ); p^\nu)
$ and we may use
Lemma \ref{Gausslemma} to evaluate it. 
For $1\leq j\leq \ell$, we have
$$
G\left(m_j, r_j \left(\overline{d}\pm 1 \right); p^\nu\right)
=\begin{cases}
p^\nu & \text{if } r_j\left(\overline{d}\pm 1\right)\equiv 0 \pmod{p^\nu},\\
0 & \text{if } r_j \left(\overline{d}\pm 1\right) \not \equiv 0 \pmod{p^\nu}.
\end{cases}
$$
For $\ell+1\leq j \leq g$, the Gauss sum equals
\begin{equation*}
\begin{aligned}
	&
	\begin{cases}
	p^{\frac{\nu+\mu_j}{2}}\varepsilon_{p^{\nu-\mu_j}}\left(\frac{m_j/p^{\mu_j}}{p^{\nu-\mu_j}}\right)e_{p^{\nu+\mu_j}}\left(-r_j^2\left(\overline{d}\pm 1\right)^2 \overline{4\frac{m_j}{p^{\mu_j}}}\right) & \text{if }r_j\left(\overline{d}\pm 1\right)\equiv 0 \pmod{p^{\nu_j}},\\
		0 & \text{if }r_j\left(\overline{d}\pm 1\right)\not\equiv 0 \pmod{p^{\nu_j}}.
	\end{cases}
	\end{aligned}
\end{equation*}
Thus $H_{m,p^\nu}^\pm(n,r)$ becomes
\begin{equation}\label{Gbound}
p^{\nu\ell}\sum_{\substack{d\pmod{p^\nu} \\ r_j\left(\overline{d}\pm1\right)\equiv0\pmod{p^\nu}(1\le j\le \ell) \\ r_j\left(\overline{d}\pm1\right)\equiv0\pmod{p^{\mu_j}}(\ell+1\le j\le g)}}\prod_{j=\ell+1}^g\varepsilon_{p^{\nu-\mu_j}}\left(\frac{m_j/p^{\nu_j}}{p^{\nu-\mu_j}}\right) p^{\frac{\nu+\mu_j}{2}}e_{p^{\nu+\mu_j}}\left(-r_j^2\left(\overline{d}\pm1\right)^2 4\overline{\frac{m_j}{p^{\mu_j}}}\right).
\end{equation}
We now consider whether $p^\nu\mid D$ or not.

If $p^\nu\mid D$, then we have $(D, p^\nu)=p^\nu$. We bound \eqref{Gbound} trivially, yielding 
$$
\left|H_{m,p^\nu}^\pm(n,r)\right| \le p^{\nu\ell} \cdot p^\nu \cdot p^{\frac{\nu}{2}(g-\ell)}p^{\frac12\sum_{j=\ell+1}^g\mu_j}=(D,p^\nu)p^{\frac{\nu g}{2}+\frac{ \nu\ell}{2}+\frac12\sum_{j=\ell+1}^g \mu_j}.
$$
Now 
$$
p^{\frac{\nu g}{2}+\frac{\nu}{2}}=c^{\frac{g+1}{2}}, \quad
p^{\frac{\nu\ell}{2}+\frac12\sum_{j=\ell+1}^g \mu_j}\le p^{\frac12\sum_{j=1}^g\mu_j}\le\det(2m)^{\frac12}, \quad
p^{-\frac{\nu}{2}}\le 1,
$$
giving the claim in this case.

If $p^\nu\nmid D$, then we use that
\begin{equation}\label{Dsplit}
D=\frac12\det(2m)\left(4n-m^{-1}[r]\right),
\end{equation}
which follows from the Jacobi decomposition.
This gives that $p^\nu$ divides at most one of the $m_j$. There are two cases to distinguish depending on whether $p^\nu$ divides one of the $m_j$ or none.

We first assume
$\nu>\mu_j$ for $1\le j\le g$ and let $\lambda:=\text{ord}_p(D)$. In (27) of \cite{BK} it was shown that
$$
\left|H_{m,c}^\pm(n,r)\right|\le 2p^{\frac{\nu(g+1)}{2}}\left(D,p^\nu\right)\ll c^{\frac{g+1}{2}}\left(D,p^\nu\right).
$$
This implies the claim in this case.

Finally we consider the case that $p^\nu$ divides exactly one $m_j$ and we may assume without loss of generality that
 $\mu_g\ge \nu$. Let
$
\kappa:=\max\{0,\nu-\rho_g,\mu_1-\rho_1,\cdots,\mu_{g-1}-\rho_{g-1}\},
$
where $\rho_j=\text{ord}(r_j)(1\le j\le g)$. It is shown in the first displayed formula on page 509 of \cite{BK} that
\begin{equation}\label{final}
\left|H_{m,c}^\pm(n,r)\right|\le p^{2\nu-\kappa}\prod_{j=1}^{g-1}p^{\frac12(\nu+\mu_j)}=p^{\frac{\nu(g+1)}{2}+\nu-\kappa}p^{\frac12\sum_{j=1}^{g-1}\mu_j}.
\end{equation}
We next analyze \eqref{Dsplit}.
Since
$
m^{-1}=\text{diag}(m_1^{-1} ,\dots ,m_g^{-1})
$, we obtain, since $p^\nu|m_g$,
$$
D=2^{g+1} n \prod_{j=1}^gm_j  - 2^{g-1}\sum_{j=1}^g r_j^2\prod_{\substack{\ell=1 \\ \ell\neq j}}^g m_j\equiv -2^{g-1} r_g^2 \prod_{j=1}^{g-1}m_j \pmod{p^\nu}.
$$
Thus, since $\lambda<\nu$,
$$
\lambda=\sum_{j=1}^{g-1}\mu_j+2\rho_g.
$$
Moreover, from the definition of $\kappa$, we obtain that $\kappa\ge \nu-\rho_g$. Thus, by \eqref{final},
$$
\left|H_{m,c}^\pm(n,r)\right|\le p^{\frac{\nu}{2}(g+1)+\frac12\sum_{j=1}^{g-1}\mu_j+\rho_g}=p^{\frac{\nu}{2}(g+1)+\frac{\lambda}{2}}=c^{\frac{g+1}{2}}D^{\frac12}\le c^{\frac{g+1}{2}}(D,c).
$$
This finishes the proof.
\end{proof}

\section{Bounding coefficients of Poincar\'e series}
In this section, we estimate the Fourier coefficients $b_{n,r}$ of $P_{k,m;(n,r)}$. This is of independent interest for obtaining bounds for Fourier coefficients of Jacobi forms.
\begin{theorem}\label{Jacobibound}
Assume that $k \in \N$ satisfies $(g+3)/2<k<g$. Then, with notation as above,
$$
b_{n,r}\left(P_{k,m;(n,r)}\right) \ll \left(1 + \frac{D^{\frac{g}{2}+\varepsilon}}{\operatorname{det}(2m)^{\frac{g+1}{2}}}\left(1 + D^{k-g-1}\operatorname{det}(2m)^{-k+g+1+\frac{1}{g-1}(-k+g+1) + \varepsilon}\right)\right).
$$
\end{theorem}
\begin{proof}
We use the explicit representation of $b_{n,r}$ given in Proposition \ref{Petprop}. The first term in \eqref{defineg} yields the first term in the bound in Theorem \ref{Jacobibound}. Thus, we have to bound
$$
f_m(n,r):=\sum_{c\geq 1}\left|H_{m,c}^\pm(n,r)\right|J_{k-\frac{g}{2}-1}\left(\frac{2\pi D}{\det(2m)c}\right)c^{-\frac{g}{2}-1}.
$$
We may rewrite
$$
f_m(n,r)=\sum_{d|D}\sum_{c\geq 1\atop \left(c,\frac{D}{d}\right)=1}\left(cd\right)^{-\frac{g}{2}-1}\left|H_{m,dc}^\pm(n,r)\right|J_{k-\frac{g}{2}-1}\left(\frac{A}{c}\right),
$$
where $A=A_d:=\frac{2\pi D}{d\det(2m)}$. To bound the inner sum,  we split it into three pieces: a part with $c\leq A$, a contribution from $A\le c\le B$, and a piece with $c\ge B$, with $B$ to be  determined later. Note that the range of any of these sums is allowed to be empty.  We require the bounds for Kloosterman sums from Section 3 as well as the following estimates
\begin{align}
\label{Kloosbound} 
\left|H_{m,c}^{\pm}(n,r)\right|&\ll c^{g+\varepsilon}(D,c), \\
\label{Jbound}
J_{\ell}(t) &\ll_{\ell} \operatorname{min}\left\lbrace t^{-\frac12}, t^{\ell}\right\rbrace. 
\end{align}
The bound \eqref{Kloosbound}  can be found in \cite{BK}  whereas \eqref{Jbound} is standard.

To bound the part with $c\le A$, we use \eqref{Kloosbound} with $dc$ instead of $c$ and the first estimate in \eqref{Jbound}. This gives
the contribution for the sum on $c$
$$
A^{-\frac12}d^{\frac{g}{2}+\varepsilon}\sum_{1\le c\le A}c^{\frac{g}{2}-\frac12+\varepsilon} 
\ll A^{\frac{g}{2}+\varepsilon}d^{\frac{g}{2}+\varepsilon}\ll \left(\frac{D}{\det(2m)}\right)^{\frac{g}{2}+\varepsilon}.
$$
Upon multiplying by $\det(2m)^{-\frac{1}{2}}$, this yields the second summand in the bound in Theorem \ref{Jacobibound}.\\
\indent
Next we estimate the part with $A\le c\le B$. For this, we use \eqref{Kloosbound} and the second estimate in \eqref{Jbound}. This gives the contribution, using that $k<g$,
$$
A^{k-\frac{g}{2}-1}d^{\frac{g}{2}+\varepsilon}\sum_{A\le c\le B}c^{-k+g+\varepsilon}\ll A^{k-\frac{g}{2}-1}d^{\frac{g}{2}+\varepsilon}B^{-k+g+1+\varepsilon}.
$$
\indent
Finally, we estimate the piece with $c\ge B$. For this, we use Lemma \ref{KloosLemma} and the second estimate in \eqref{Jbound}. This gives the bound
\begin{align*}
A^{k-\frac{g}{2}-1}d^{\frac{1}{2}}\det(2m)^{\frac12}\sum_{c\ge B}c^{\frac{g+1}{2}-k}
&\ll B^{\frac{g+3}{2} -k}A^{k-\frac{g}{2}-1}d^{\frac{1}{2}}\det(2m)^{\frac12}
\end{align*}
since $k>(g+3)/2$.

Now, to minimize the error, we choose B such that the second and third error agree (up to $\varepsilon$ exponents). One can show that this is the case for
\begin{align*}
	&B=d^{-1}\det(2m)^{\frac{1}{g-1}}.
\end{align*}
Plugging back in gives the claim after multiplying by $\det(2m)^{-\frac{1}{2}}$.
\end{proof}
\section{Proof of theorem \ref{maintheorem}}
In this section we use the previous bounds with $ g \mapsto g-1$ and $\phi = \phi_m$, where $\phi_m$ comes from the Jacobi coefficients of a Siegel modular form. We recall the following bound from Proposition 2 of \cite{BK}.
\begin{lemma}\label{Phibound} 
If $\phi_m$ is the $m$th Fourier-Jacobi coefficient of a Siegel modular form $F$, then
$$
||\phi_m|| \ll_{\varepsilon, F} \operatorname{det}(2m)^{\frac{k}{2}-\alpha_g + \varepsilon}\qquad (\varepsilon >0).
$$
\end{lemma}
We are now ready to prove Theorem \ref{maintheorem}
\begin{proof}[Proof of Theorem \ref{maintheorem}]
By Lemma 2.2 and Lemma \ref{Phibound}, 
\[
 a(T) \ll \left\lvert b_{n,r} \left( P_{k,m ; (n,r)} \right) \right\rvert^{\frac{1}{2}} D^{\frac{k}{2} -\frac{g}{4} -\frac{1}{4}} \det (2m)^{\frac{g}{4} +\frac{1}{2} -\alpha_g +\varepsilon} .
\]
Theorem \ref{Jacobibound} then yields
\begin{equation} \notag 
A(T) \ll \left( \operatorname{det}(2m)^{-\frac{g}{2}}f(m,D)\right)^{\frac{1}{2}} D^{\frac{k}{2}-\frac{g}{4}-\frac{1}{4}}\operatorname{det}(2m)^{\frac{g}{4}+\frac12-\alpha_g +\varepsilon},
\end{equation}
where
$$
f(m,D):=\operatorname{det}(2m)^{\frac{g}{2}} + D^{\frac{g-1}{2} + \varepsilon}\left(1 +D^{k-g}\operatorname{det}(2m)^{-k+g+\frac{1}{g-2}(-k+g)+\varepsilon}\right).
$$
Define 
$$
m_{g-1}(T) := \min\left\lbrace T[U]|_{g-1} \mid U \in \operatorname{GL}_g(\Z) \right\rbrace, 
$$
where $T[U]|_{g-1}$ denotes the determinant of the leading  $(g-1)$-rowed submatrix of $T[U]$. Since both sides of the bound in Theorem \ref{maintheorem} are invariant under replacing $T$ by $T[U]$ ($U \in \operatorname{GL}_g(\Z)$), we may assume that $T=\left(\begin{smallmatrix} n & r^T/2 \\ r/2 & m\end{smallmatrix} \right)$ with $\det(m) = m_{g-1}(T)$.
Now, by reduction theory,
\begin{equation*}
\notag 
\operatorname{det}(m)=m_{g-1}(T) \ll D^{1-\frac{1}{g}}.
\end{equation*}

It is easy to see that 
the powers of $\operatorname{det}(m)$ in $f(m,D)$ are all non-negative. So we may replace $\operatorname{det}(m)$ by $D^{1-\frac{1}{g}}$ in this expression. One can then show that 
the last term in $f(m, D)$ is dominant.  This gives
$$
	a(T) 
	\ll\det(2m)^{\frac{1}{2}-\alpha_g+\varepsilon}D^{\frac{k}{2}-\frac{1}{2}+\frac{g-k}{2g(g-2)}+\varepsilon}.
$$
This yields the claim of the theorem since $1/2 -\alpha_g >0$.
\end{proof}

\end{document}